\documentclass[11pt, a4paper,twoside]{amsart}
\usepackage[utf8]{inputenc}
\usepackage[T1]{fontenc}
\usepackage[english]{babel}
\usepackage{amsmath,amsfonts,amssymb,amsthm}
\usepackage{stmaryrd} 
\usepackage{array}
\usepackage{enumerate}
\usepackage{hyperref}
\usepackage{mathabx}
\usepackage{comment}
\usepackage{ifthen}

\newcommand{\Z}{{\mathbb Z}}
\newcommand{\Q}{{\mathbb Q}}
\newcommand{\R}{{\mathbb R}}
\newcommand{\Qb}{\overline{\mathbb Q}}
\newcommand{\C}{{\mathbb C}}
\newcommand{\A}{{\mathcal A}}
\newcommand{\e}{{\varepsilon}}
\newcommand{\m}[1]{{\overline{h}(#1)}} 
\newcommand{\h}[1]{{\overline{h}^*\!(#1)}} 
\newcommand{\Card}[1]{{\vert #1\vert}} 
\newcommand{\Gal}{\mathrm{Gal}}

\newcommand{\poubelle}[1]{}
\newcommand{\SGal}{{\overline{S}^{\Gal}}}

\theoremstyle{plain}
\newtheorem{theorem}{Theorem}[section]

\newtheorem{lemma}[theorem]{Lemma}
\newtheorem{remark}[theorem]{Remark}

\newtheorem{corollary}[theorem]{Corollary}

\begin{document}

\title[On Bilu's equidistribution theorem]{Equidistribution for sets which are not necessarily Galois stable: On a theorem of Mignotte}
\date{}
\author{Francesco Amoroso and Arnaud Plessis}
\address{Francesco Amoroso: Laboratoire de Math\'ematiques N.~Oresme, Université de Caen Normandie, CNRS UMR 6139,  BP 5186, 14032 Caen Cedex, France} 
\address{Arnaud Plessis: Academy of Mathematics and Systems Science, Morningside Center of Mathematics, Chinese Academy of Sciences, Beijing 100190, China}

\maketitle 

\begin{abstract}
An important result of Bilu deals with the equidistribution of the Galois orbits of a sequence $(\alpha_n)_n$ in $\Qb^*$. Here, we prove a quantitative  equidistribution theorem for a sequence of finite subsets in $\Qb^*$ which are not necessarily stable by Galois action. We follow a method of Mignotte.
\end{abstract}


\section{Introduction}
Let $X$ be a metric space.
For a finite subset $T\subset X$, the discrete probability measure on $X$ associated to it is given by \[ \mu_{T,X} = \frac{1}{\Card{T}}\sum_{\alpha\in T} \delta_{\alpha,X},\]
where $\Card T$ denotes the cardinality of $T$ and $\delta_{\alpha,X}$ the Dirac measure on $X$ supported at $\alpha$.
In the special case $X=\C^*$, we put $\mu_T=\mu_{T,\C^*}$.

We say that a sequence $(\mu_n)_n$ of probability measures on $X$ \textit{converges in distribution} to $\mu$ if for every bounded continuous function $f : X\to \C$, we have \[ \lim_{n\to +\infty} \int_X fd\mu_n = \int_X fd\mu. \]
An important example of such a sequence was given by Bilu, see Theorem \ref{thm-Bilu} below.
If the limit above holds for all compactly supported continuous functions, we say that $\mu_n$ \textit{weakly converges} to $\mu$. 

Throughout this text, we define $h:\Qb\to \R$ to be the (logarithmic, absolute) Weil height and $\lambda$ to be the Haar probability measure on the complex unit circle.  We also denote by $\mu_\infty$ the set of roots of unity in $\Qb$.

\begin{theorem} [Bilu, \cite{Bilu}] \label{thm-Bilu}
Let $K$ be a number field, and let $(\alpha_n)_n$ be a sequence of $\Qb^*$ such that $h(\alpha_n)\to 0$ and $[K(\alpha_n) : K]\to +\infty$. 
Then $\mu_{S_n}$ converges in distribution to $\lambda$, where $S_n$ is the Galois orbit of $\alpha_n$ over $K$.
\end{theorem}

This theorem (which was originally formulated with $K=\Q$) was inspired by a previous work of Szpiro, Ullmo and Zhang who studied the equidistribution of points of small N\'eron-Tate height on abelian varieties \cite{SUZ}. These two well-known results have been largely generalized to other heights and places, see for instance \cite{Zhang, Rumely, BakerHsia, BakerPetsche, Petsche, FavreLetelier, BakerRumely, ChambertLoir, Yuan, Gubler, Dandrea, Burgos, Demarco, BakerMasser}.
Roughly speaking, each one of these results contains an equidistribution theorem \textit{\`a la} Bilu, that is, a statement of the form ``Let $K$ be a number field, and let $S_n$ be a $\Gal(\overline{K}/K)$-invariant subset of $X$ such that the average of height of $\alpha$, with $\alpha\in S_n$, goes to $0$ and $\Card{S_n}\to +\infty$. Then $\mu_{S_n,X}$ weakly converges to some Haar measure". 

Here we prove an equidistribution theorem for a sequence of finite subsets in $\Qb^*$ which are not necessarily stable by Galois action. 
After posting this paper on ArXiv, Fili informed us that the qualitative version (but not the quantitative one) of our result partially follows from a recent preprint~\cite{Fili}, of which we were unaware. 

The first avatar of equidistribution theorems is a result of Langevin~\cite{Langevin}: Given an open set $\Delta\subset\C$ intersecting the unit circle, the algebraic integers whose Galois conjugates all lie outside $\Delta$ cannot have a Weil height arbitrarily small. The idea of Langevin's proof is to show that the transfinite diameter of the complement of $\Delta$ in the unit disk has transfinite diameter less than $1$. He then concludes by applying a theorem of Fekete~\cite{Fekete} which asserts that there are only a finite number of algebraic integers whose Galois conjugates all belong to a set of transfinite diameter $<1$.

Soon after, Mignotte~\cite{Mignotte} gives an entirely different proof of Langevin's result, see~\cite[Chapter 15]{Masser} for an excellent expository text. Equidistribution theorems, stated in the modern language of weak convergence of probability measures, follow from radial and angular distribution of Galois conjugates of algebraic numbers with small height. 
Here, radial distribution means that ``most of" these conjugates have absolute value close to $1$. 
It is easily established from the definition of the Weil height. 
The angular distribution deals with the distribution in $[0, 2\pi]$ of arguments of these conjugates modulo $2\pi$.
It is the hardest part of Mignotte's proof. The first idea is to apply a result of Erdős and Turán \cite{Erdos} (see Theorem~\ref{thmET} in Section~\ref{section preli}) which asserts that the arguments of roots of a polynomial $Q(X)=\sum_{i=0}^D q_iX^i$ with complex coefficients are well distributed in $[0, 2\pi]$ if the ratio 
\begin{equation}
\label{ET}
\frac{\mathcal{L}(Q)}{\sqrt{\vert q_Dq_0\vert}}
\end{equation}
is ``not too large". Here, $\mathcal{L}(Q)=\sum_{i=0}^D \vert q_i\vert$ denotes the length of $Q$. Unfortunately, the coefficients of the minimal polynomial of an algebraic number $\alpha$ can be very large, even if the height of $\alpha$ is small. This is when the second ingredient of the proof comes in. A classical result in diophantine approximation, the Siegel Lemma, shows that there exists a polynomial $Q\in\Z[X]$ of ``small" degree vanishing at $\alpha$  such that $q_0q_D$ is non-zero and its coefficients are ``not too large" if the height of $\alpha$ is small. The quantity in~\eqref{ET} is therefore ``not too large" since $\vert q_0q_D\vert\geq1$.

In the nineties, Mignotte pointed out to the first author that to obtain the angular distribution, we do not need to have an auxiliary polynomial with integer coefficients, but only with $\vert q_0q_D\vert$ not too small. This innocent remark was one of the starting points of our investigation.
 
In this article, we apply Mignotte's method to deal with the radial and angular distribution of sets $S$ with large cardinality and made of algebraic numbers with small height. The novelty is that our sets are no more assumed to be stable by Galois action.
This prevents us from using the standard Siegel Lemma, which makes the study of the angular distribution more complicated. Fortunately, there is an absolute version of this lemma, which follows from deep results of Zhang \cite[Theorem 1.10]{Zhang2}.
This {\sl absolute Siegel Lemma} can be applied to our situation, but it provides us an auxiliary polynomial $Q\in\overline{\Q}[X]$  whose coefficients (which can be assumed to be algebraic integers) cannot be controlled. 
In particular, $\vert q_0q_D\vert$ can be as small as possible, but its absolute norm has to be a positive integer. This naturally leads us to consider the {\sl arithmetic mean} of the radial and angular discrepancy of the conjugate sets of $S$. We prove that both of them are small. 

Let $r>1$ be a real. 
We write $\A_r$ for the closed annulus centred at the origin with inner radius $1/r$ and outer radius $r$. We also denote by $\m{S}$ the (arithmetic) mean of $h(\alpha)$ with $\alpha\in S$, that is, 
\[\m{S}= \frac{1}{\Card{S}}\sum_{\alpha\in S} h(\alpha).\]
We finally set
\[\h{S}= 24\left(\m{S}+\frac{\log(2\Card{S})}{\Card{S}}\right)^{1/3}.\]

\begin{theorem} 
\label{thm-radial-angular}
Let $S\subset \overline{\Q}^*$ be a finite subset. 
\begin{itemize}
\item[(1)] For any $r>1$, we have 
\[\frac{1}{[\Q(S):\Q]}\sum_{\sigma\colon\Q(S) \hookrightarrow \C} \frac{\Card{\sigma(S)\backslash \A_r}}{\Card{S}} \leq \frac{2\m{S}}{\log r}.\]
\item[(2)] For any sector $\Delta$ of angle $\theta\in[0,2\pi]$ based at the origin, we have
\[\frac{1}{[ \Q(S) : \Q]}\sum_{\sigma\colon\Q(S) \hookrightarrow \C} \left\vert\frac{\Card{\sigma(S)\cap\Delta}}{\Card{S}}- \frac{\theta}{2\pi}\right\vert \leq  \h{S}. \]
\end{itemize}
\end{theorem}

When $S$ is Galois invariant, we recover Mignotte's results~\cite{Mignotte}. The following example is a good illustration of what happens when we drop the assumption of Galois invariance on $S$. Let $p$ be a prime number, and let $\zeta_p=\exp(2i\pi/p)$. We choose $S=S_p$ as the set of $p$-roots of unity $\zeta_p^k$ with $1\leq k\leq[\sqrt{p}]$, where $[x]$ denotes as usual the integer part of $x\in\R$. Thus 
\[\h{S_p}= 24\left(\frac{\log(2[\sqrt{p}])}{[\sqrt{p}]}\right)^{1/3}\to 0\] 
as $p\to\infty$. We fix a sector $\Delta$ of angle $\theta\in[0,2\pi]$ based at the origin with $1\not\in\Delta$. Then  
$S_p\cap\Delta$ is empty when $p$ is sufficiently large; whence
\[\left\vert\frac{\Card{S_p\cap\Delta}}{\Card{S_p}}- \frac{\theta}{2\pi}\right\vert=\frac{\theta}{2\pi}.\]
However, by Theorem~\ref{thm-radial-angular}(2) (and as it can be directly verified), 
\[\frac{1}{p-1}\sum_{\sigma\colon\Q(\zeta_p) \hookrightarrow \C}\left\vert\frac{\Card{\sigma(S_p)\cap\Delta}}{\Card{S_p}}- \frac{\theta}{2\pi}\right\vert\to 0\]
as $p\to\infty$.
This means that there are a "small" number of $\Q$-embeddings $\sigma : \Q(\zeta_p) \to \C$ for which the ratio $\Card{\sigma(S_p)\cap\Delta}/\Card{S_p}$ is "far" from $\theta/(2\pi)$. \\

Theorem~\ref{thm-radial-angular} allows us to prove new results on points of small Weil height. Following Bombieri and Zannier~\cite{Bombieri}, we say that a set $S$ of algebraic numbers has the {\it Bogomolov property}, or short property $(B)$, if there exists a positive constant $c$ such that the Weil height of an element in $S$ is either $0$ or bounded from below by $c$. Property $(B)$ was established for the maximal totally real extension $\Q^{\rm tr}$ of $\Q$ by Schinzel, see~\cite{Schinzel}. 
Note that it is also an immediate consequence of Theorem \ref{thm-Bilu}. The following corollary can be viewed as a generalization of this result.
\begin{corollary}
\label{KQtr}
Let $L$ be an algebraic field. Then for any $\alpha$ in the group product $L^*(\Q^{\rm tr})^*$ such that $h(\alpha)<5\cdot 10^{-6}$, we have 
$[L(\alpha):L]<4\cdot 10^6$.
\end{corollary}
\begin{proof}
Assume by contradiction that there exists $\alpha\in L^*(\Q^{\rm tr})^*$ such that $h(\alpha)<5\cdot 10^{-6}$ and $d=[L(\alpha):L]\geq 4\cdot 10^6$. We write $\alpha = y z$ with $y\in L^*$ and $z\in (\Q^{\rm tr})^*$. Define $S$ as the Galois orbit of $\alpha$ over $L$. Thus, for any $\Q$-embedding $\sigma\colon\Q(S)\rightarrow\C$, we have $\sigma(S)\subseteq \sigma(y)\R$. Hence, for any sector $\Delta$ based at the origin, 
\[\frac{\Card{\sigma(S)\cap\Delta}}{\Card{S}}=
\begin{cases}
1 &\hbox{ if } \sigma(y)\in\Delta;\\
0 &\hbox{ otherwise.}
\end{cases}\] 
Thus, choosing for $\Delta$ any sector of angle $\pi$, the value of the left-hand side in Theorem~\ref{thm-radial-angular}(2) is $1/2$. On the other hand, $\Card{S}=d$ and $\log(2d)/d< 4\cdot 10^{-6}$ since $d\geq 4\cdot 10^6$ by assumption. Moreover, $\m{S}=h(\alpha)<5\cdot 10^{-6}$.
In conclusion, $\h{S}= 24(\m{S}+\log(2\Card{S})/\Card{S})^{1/3}<1/2$, a contradiction.
\end{proof}
The maximal abelian extension $\Q^{\rm ab}$ of $\Q$ also satisfies property $(B)$, as it has been conjectured by Zannier and proved in~\cite{AmorosoDvornicich}. However, the compositum of $\Q^{\rm tr}$ and $\Q^{\rm ab}$ does not satisfy property $(B)$ since its subfield $\Q^{\rm tr}(i)$ does not satisfy $(B)$, see~\cite[Theorem 5.3]{AmorosoDavidZannier}. Nevertheless, 
\begin{corollary}
\label{QabQtr}
The group product $G=(\Q^{\rm ab})^*(\Q^{\rm tr})^*$ satisfies property $(B)$. 
\end{corollary}
\begin{proof}
Assume by contradiction that there exists a sequence $(\alpha_n)_n$ belonging to $G\backslash\mu_\infty$ such that $h(\alpha_n)\to 0$.  
By~\cite[Theorem 1.1]{AmorosoZannier}, $[\Q^{\rm ab}(\alpha_n) : \Q^{\rm ab}]\to +\infty$. We then apply Corollary~\ref{KQtr} with $L=\Q^{\rm ab}$. 
\end{proof}

Theorem~\ref{thm-radial-angular} can be formulated in terms of convergence in distribution, as it was expected in~\cite{Masser} 
(see the paragraph therein around equation (15.9)). Let $r>1$, and let $f: \C^*\to \C$ be a function that is Lipschitz on $\A_r$. We define $\mathrm{Lip}_r(f)$ to be the Lipschitz constant of $f$ on $\A_r$. 
The infinite norm of $f$ on a set $T\subset \C^*$ is denoted with $\lVert f\rVert_{\infty,T}$.
Finally, given a subset $S\subset\Qb^*$, we define $\SGal$ as the smallest $\Gal(\overline{\Q}/\Q)$-invariant set containing $S$.

\begin{theorem} 
\label{main thm}
Let $S\subset \Qb^*$ be a finite set, let $r>1$ be a real number, and let $N\geq 2$ be an integer.
\begin{itemize}
\item[(1)] For all functions $f: \C^* \to \C$ that are Lipschitz on $\A_r$, we have 
\begin{multline*}
\frac{1}{[ \Q(S) : \Q]}\sum_{\sigma\colon\Q(S) \hookrightarrow \C}\left\vert \int_{\C^*} fd \mu_{\sigma(S)} - \int_{\C^*} fd\lambda\right\vert\\
\leq\frac{4r\pi \mathrm{Lip}_r(f)}{N}+(\Vert f\Vert_{\infty,\SGal\backslash \A_r}+ 2\lVert f\rVert_{\infty,\A_r})\frac{2\m{S}}{\log r}
+ 2N\lVert f\rVert_{\infty,\A_r}\h{S}.
\end{multline*}
\item[(2)] Let $\e\in (0,1)$, and let $L$ be a number field. Then there exists a set $\Lambda=\Lambda(S,r,N,\e,L)$ of $L$-embeddings $L(S)\hookrightarrow\C$ with cardinality at least $(1-\e)[L(S):L]$ such that
\begin{multline*}
\left\vert \int_{\C^*} fd \mu_{\sigma(S)} - \int_{\C^*} fd\lambda\right\vert
\leq\frac{4r\pi \mathrm{Lip}_r(f)}{N}\\
+(\Vert f\Vert_{\infty,\SGal\backslash \A_r}+ 2\lVert f\rVert_{\infty,\A_r})\frac{4[L:\Q]\m{S}}{\e\log r}
+ \frac{4[L:\Q]N^2}{\e}\lVert f\rVert_{\infty,\A_r}\h{S}
\end{multline*}
for all $\sigma\in \Lambda$ and all functions $f: \C^* \to \C$ that are Lipschitz on $\A_r$. 
\end{itemize}
\end{theorem}

Theorem \ref{main thm} implies a quantitative version of Bilu's equidistribution theorem.  Such versions already exist in the literature, see \cite{Petsche, FavreLetelier, Dandrea, BakerMasser}. In all these articles, the estimations are stronger than ours, but they can only hold for much more restrictive functions. For instance, in the first three references, $f$ must be at least bounded and differentiable on $\C^*$ and in the last one, $f$ has to be Lipschitz on $\A_r$, continuous on $\C^*$ and satisfy $\vert f(z)\vert \leq \log \vert z\vert$ when $\vert z\vert \geq r$ as well as $\vert f(z)\vert \leq \log \vert z\vert^{-1}$ when $0< \vert z\vert \leq 1/r$. 
Regarding the test functions of Theorem \ref{main thm}, they must be Lipschitz on $\mathcal{A}_r$, but can be unbounded and totally discontinuous outside. 

Theorem \ref{main thm} has the following qualitative consequence:

\begin{corollary} 
\label{cor gen Bilu}
Let $(S_n)_n$ be a sequence of finite subsets of $\Qb^*$ such that $\Card{S_n} \to +\infty$ and $\m{S_n}\to 0$. Let $V\subset \C^*$ be a neighbourhood of the unit circle. We consider the class of test functions $f : \C^* \to \C$ satisfying
\begin{equation}
\label{test}
\hbox{$f$ is continuous on} \; V \quad \hbox{and}\quad \m{S_n}\Vert f\Vert_{\infty,S_n^{\Gal}\backslash V}\to 0, \hbox{ as } n\to\infty.
\end{equation}
Then
\begin{itemize}
\item[(1)] We have 
\[\frac{1}{[ \Q(S_n) : \Q]}\sum_{\sigma\colon\Q(S_n) \hookrightarrow \C}\left\vert \int_{\C^*} fd \mu_{\sigma(S_n)} - \int_{\C^*} fd\lambda\right\vert
\rightarrow0\]
for all functions $f: \C^*\to \C$ satisfying~\eqref{test}.
\item[(2)] Let $\e \in (0,1)$, and let $L$ be a number field. 
Then for all integers $n\geq1$, there is a set $\Lambda_n$ of $L$-embeddings $L(S_n)\hookrightarrow\C$ with cardinality at least $(1-\e)[L(S_n):L]$ such that for all $(\sigma_n)_n\in\prod_n \Lambda_n$, we have 
\begin{equation*}
\frac{1}{\Card{S_n}}\sum_{\beta\in \sigma_n S_n} f(\beta) \to \int_0^1 f(e^{2i\pi t})dt 
\end{equation*}
for all functions $f: \C^*\to \C$ satisfying~\eqref{test}.
\end{itemize}
\end{corollary}
Note that every bounded continuous function $f \colon\C^*\to \C$ satisfies~\eqref{test}. Thus, the second assertion of this theorem implies:

\begin{corollary} 
\label{cor Fili}
Let $(S_n)_n$ be a sequence of finite subsets of $\Qb^*$ such that $\Card{S_n} \to +\infty$ and $\overline{h}(S_n) \to 0$.
Let $L$ be a number field. 
Then for all integers $n\geq 1$, there is a set $\Lambda_n$ of $L$-embeddings $L(S_n)\hookrightarrow\C$ with cardinality at least $(1-\e)[L(S_n):L]$ such that for all $(\sigma_n)_n\in\prod_n \Lambda_n$, the sequence of discrete probability measures $\mu_{\sigma_n S_n}$ converges in distribution to $\lambda$.
\end{corollary}

Corollary~\ref{cor Fili} partially follows from \cite[Theorem~3.17]{Fili}. 
More precisely, under the same assumptions on $S_n$ (which corresponds to a very special case of \cite[Theorem~3.17]{Fili}), Doyle, Fili and Tobin obtained the same conclusion than ours, but for the weak convergence.
Their proof is based on potential theory, which is the other classical approach to deal with equidistribution.

Theorem \ref{thm-Bilu} does not hold anymore if we relax the assumption ``$K$ is a number field" to ``$K$ is an algebraic field". Indeed, put $\alpha_n=(1-e^{2i\pi/n^n})^{1/n}$. Let $S_n$ be the Galois orbit of $\alpha_n$ over $\Q(\mu_\infty)$. Clearly, $h(\alpha_n)\leq (\log 2)/n$ and $\alpha_n\notin\mu_\infty$ for all integers $n\geq1$. 
Hence, $h(\alpha_n)\to 0$ and $\Card{S_n} \to +\infty$. 
Moreover, each element of $S_n$ has absolute value $\vert \alpha_n\vert$. 
The series expansion of the exponential implies $\alpha_n\to 0$, and so the sequence $\mu_{S_n}$ cannot converge in distribution to $\lambda$.  
This elementary example shows that the conclusion of the next corollary is somehow optimal. 
\begin{corollary} 
\label{cor obvious}
Let $K$ be an algebraic field, and let $(\alpha_n)_n$ be a sequence of $\Qb^*$ such that $h(\alpha_n)\to 0$ and $[K(\alpha_n) : K]\to +\infty$. 
Let $L$ be a number field.
Then there is a set $\Lambda_n$ of $L$-embeddings $L(S_n)\hookrightarrow\C$ with cardinality at least $(1-\e)[L(S_n):L]$ such that for all $(\sigma_n)_n\in\prod_n \Lambda_n$, the sequence of discrete probability measures $\mu_{\sigma_n S_n}$ converges in distribution to $\lambda$, where $S_n$ is the Galois orbit of $\alpha_n$ over $K$. 
\end{corollary}
\begin{proof}
Take for $S_n$ the set of conjugates of $\alpha_n$ over $K$ in Corollary~\ref{cor Fili}.
\end{proof}

\subsection*{Plan of the article}
The plan of the article is as follows. In Section~\ref{section preli}, we implement our generalisation of Mignotte's method to treat the radial and angular distributions, then we prove Theorem~\ref{thm-radial-angular}. In Section~\ref{main}, we deduce Theorem~\ref{main thm} from Theorem~\ref{thm-radial-angular} and Corollary~\ref{cor gen Bilu} from Theorem~\ref{main thm}.

\subsection*{Acknowledgement}
We thank Fili, Sombra and Weiss for their fruitful discussions. The second author sincerely thanks the \textit{Laboratoire de Math\'ematiques N. Oresme} of \textit{Universit\'e de Caen Normandie} for housing him on May 2023 so that he can work face-to-face with the first author. 
He was also funded by the Morningside Center of Mathematics, CAS.  

\section{Radial and angular distribution} \label{section preli}
The proof of Theorem~\ref{thm-radial-angular} is based on arguments due to Mignotte, which are well highlighted in \cite[Chapter 15]{Masser}. We will follow the exposition of this book. Throughout this section $r$ denotes a real number greater than $1$ and $S\subset \overline{\Q}^*$ a finite set.

\subsection{Radial distribution} \label{section radial}
The goal of this subsection is to establish that the mean of the number of elements belonging to $\sigma(S)\backslash \A_r$, where $\sigma$ ranges over all $\Q$-embeddings $\Q(S) \hookrightarrow \C$, cannot be ``too large" if $\m{S}$ is small. Recall that $\delta_x$ is the Dirac measure on $\C^*$ supported at $x$.
Let us start by the following lemma. 

\begin{lemma}[\cite{Masser}, Theorem 15.1] \label{lmm Masser}
For all $\alpha\in\Qb^*$, we have \[ \sum_{\tau : \Q(\alpha)\hookrightarrow\C} \delta_{\tau\alpha}(\A_r) \geq [\Q(\alpha) : \Q]\left(1-\frac{2h(\alpha)}{\log r}\right).\]
\end{lemma}
We can now prove the ``radial part" of Theorem~\ref{thm-radial-angular}.
\begin{proof}[Proof of Theorem~\ref{thm-radial-angular}(1)]
Since $\Card{\sigma(S)\cap \A_r}+\Card{\sigma(S)\backslash \A_r}=\Card{S}$, the inequality in Theorem~\ref{thm-radial-angular}(1) is equivalent to \[ \frac{1}{[\Q(S):\Q]}\sum_{\sigma\colon\Q(S) \hookrightarrow \C} \Card{\sigma(S)\cap \A_r} \geq \Card{S}\left(1-\frac{2\m{S}}{\log r}\right),\] 
which we now prove. The set $\sigma(S)\cap\A_r$ has cardinality $\sum_{\alpha\in S} \delta_{\sigma\alpha}(\A_r)$ for all $\sigma\colon\Q(S)\hookrightarrow \C$. 
Thus, \[\sum_{\sigma\colon\Q(S)\hookrightarrow \C} \Card{\sigma(S)\cap\A_r} = \sum_{\alpha\in S}\sum_{\sigma\colon\Q(S)\hookrightarrow \C} \delta_{\sigma\alpha}(\A_r).\]  
Then, each $\Q$-embedding $\tau : \Q(\alpha) \hookrightarrow \C$ can be extended in $[\Q(S) : \Q(\alpha)]$ different ways to a $\Q$-embedding from $\Q(S)$ to $\C$, which leads to 
\[\sum_{\sigma\colon\Q(S)\hookrightarrow \C} \Card{\sigma(S)\cap\A_r}= \sum_{\alpha\in S} [ \Q(S) : \Q(\alpha)]\sum_{\tau : \Q(\alpha)\hookrightarrow \C} \delta_{\tau\alpha}(\A_r). \]
By Lemma~\ref{lmm Masser}, the right-hand side is at least $[\Q(S) : \Q]\sum_{\alpha\in S} \left(1-\frac{2h(\alpha)}{\log r}\right)$. 
Theorem \ref{thm-radial-angular}(1) now arises from a small calculation.
\end{proof}

\subsection{Angular distribution} \label{section angulaire}
This subsection aims to show that the elements of $\sigma(S)$ are, in average, angularly well distributed when $\h{S}$ 
is small.
Concretely, if $\Delta$ is a sector of angle $\theta$ based at the origin, then the mean of number of elements belonging to $\sigma(S)\cap\Delta$, where $\sigma$ runs over all $\Q$-embeddings $\Q(S)\to \C$, is approximately $\theta\Card{S}/(2\pi)$ when $\h{S}$ is small.   


\begin{remark}
\label{remark-angular}
The left-hand side in Theorem~\ref{thm-radial-angular}(2) is obviously bounded from above by $1$. 
Hence, the theorem is trivial unless $\h{S}\leq 1$. Moreover, if  $\h{S}\leq 1$, then the definition of $\h{S}$ implies $24(\log(2\Card{S})/\Card{S})^{1/3}\leq 1$, and so $\Card{S}\geq 3$. 
To summarize, we can reduce the proof of Theorem~\ref{thm-radial-angular}(2) to the case that $\h{S}\leq 1$ and $\Card{S}\geq 3$, what we now assume.
\end{remark}

The proof of Theorem~\ref{thm-radial-angular}(2) is mainly based on two ingredients.  
The first one is a result due to Erdős and Turán \cite{Erdos}, see also \cite{AmorosoMignotte} for a more modern proof of something slightly sharper. 

For any region $\Delta\subset\C$ and any polynomial $Q\in \C[X]$, we denote by $Z_{\Delta,Q}$ the number of zeroes of $Q$ (with multiplicity) lying in $\Delta$. 

\begin{theorem} [Erdős-Turán] \label{thmET}
Let $\Delta$ be a sector of angle $\theta\in[0,2\pi]$ based at the origin, and let $Q(X)=\sum_{i=0}^D q_iX^i\in \C[X]$ be a polynomial with $q_Dq_0\neq 0$. Then
\[ \left\vert Z_{\Delta, Q}- \frac{\theta}{2\pi} D \right\vert^2 \leq 256D\log\left(\frac{\mathcal{L}(Q)}{\sqrt{\vert q_Dq_0\vert}}\right),  \] 
where $\mathcal{L}(Q)=\sum_{i=0}^D \vert q_i\vert$ denotes the length of $Q$. 
\end{theorem}

Write $P_S\in\Q(S)[X]$ for the polynomial $\prod_{\alpha\in S} (X-\alpha)$. Thus $\Card{\sigma(S)\cap\Delta}=Z_{\Delta,\sigma P_S}$.
The natural idea to get Theorem~\ref{thm-radial-angular}(2) would be to apply Theorem~\ref{thmET} to $Q=\sigma P_S$ with $\sigma$ running over all $\Q$-embeddings $\Q(S)\hookrightarrow\C$. 
But the mean of $\log(\mathcal{L}(\sigma P_S))$ might be too large, spoiling our chances of getting what we wish.
This is when the second ingredient comes in: the absolute Siegel's lemma.

Let $F(X)=\sum_{i=0}^L f_iX^i\in \Qb[X]$ be a polynomial. The height of $F$, denoted with $h(F)$, is the Weil height of coefficients of $F$, that is, 
\[ h(F) = \frac{1}{[E : \Q]} \sum_v [E_v : \Q_v]\log\max\{\vert f_0\vert_v,\dots,\vert f_L\vert_v\}, \] 
where $E$ is any number field containing $f_0,\dots,f_L$ and where $v$ ranges over all places of $E$. It is well-known that this definition does not depend on the choice of such a field $E$. 
\begin{theorem}[Absolute Siegel's Lemma]
\label{Absolute Siegel's Lemma}
Let $L$ be a positive integer with $L>\Card{S}$. Then there exists a non-zero polynomial $F$, with algebraic integer coefficients and degree $<L$, vanishing at $S$ such that:
\[
h(F)\leq  \frac{\Card{S}}{L+1-\Card{S}}\left(\frac{3}{2}\log(L+1)+(L+1)\m{S}\right) + \frac{\log(L+1)}{2}.
\]
\end{theorem}
This statement improves the main result of Roy and Thunder, see \cite[Theorem 2.2]{RoyThunder}. It is an easy consequence of~\cite[Theorem 5.2]{Zhang2}, see~\cite[Proposition 4.2]{AmorosoZannier} for details\footnote{In {\sl op.cit.} the relevant height is $h_2(F)$, with the $L_2$-metric at archimedean places, but obviously $h(F)\leq h_2(F)$.}. 

Define $L$ as the round up to $(1+\h{S}/6)\Card{S}$. By Remark~\ref{remark-angular}, we have $L\leq 2\Card{S}-1$.
By Theorem~\ref{Absolute Siegel's Lemma} and using the inequalities
\[\left(1+\frac{\h{S}}{6}\right)\Card{S} \leq L+1\leq 2\Card{S},\]
 we find a non-zero polynomial $F(X)=\sum_{i=0}^D f_i X^i$, with algebraic integer coefficients and degree $<L$, divisible by $P_S$ such that 
\begin{equation} \label{eq angular 1}
h(F) \leq \frac{6}{\h{S}}\left(\frac{3}{2}\log(2\Card{S})+2\Card{S}\m{S}\right)+ \frac{\log(2\Card{S})}{2}.
\end{equation}
Dividing $F$ by a power of $X$ if needed, we can assume that $f_Df_0\neq 0$. We now choose a number field $E$ containing $\Q(S)$ and all coefficients of $F$. 

To prove Theorem~\ref{thm-Bilu} via this approach (see \cite[Chapter 15]{Masser}), $S$ is the set of all Galois conjugates of some $\alpha\in\Qb^*$. Thus, $P_S$ is the minimal polynomial of $\alpha$ over $\Q$ and the classical Siegel's lemma asserts that we can take $F$ with integer coefficients, which implies $\vert f_Df_0\vert\geq 1$. 
Unfortunately, in our situation, we can have $\vert f_Df_0\vert\neq 0$ as small as possible. 
Nonetheless, $f_Df_0$ is an algebraic integer and therefore its norm $\prod_{\sigma\colon E\hookrightarrow  \C} \vert \sigma(f_Df_0)\vert $ over $\Q$ is a positive integer. 

\begin{lemma} \label{lmm inter}
Let $\Delta$ be a sector of angle $\theta\in[0,2\pi]$ based at the origin. 
Then \[ \frac{1}{[E:\Q]}\sum_{\sigma\colon E \hookrightarrow  \C} \left\vert Z_{\Delta,\sigma F}- \frac{\theta}{2\pi}D \right\vert^2 \leq \left(\frac{2\Card{S}\h{S}}{3}\right)^2.\]
\end{lemma}

\begin{proof}
Let $\sigma\colon E \hookrightarrow \C$ be a $\Q$-embedding. 
Then,
\[\mathcal{L}(\sigma F) = \sum_{i=0}^D \vert \sigma f_i\vert \leq (D+1)\max\{\vert \sigma f_0\vert, \dots, \vert \sigma f_D\vert\}.\]
By~\eqref{eq angular 1}, we get
\begin{align*}
\frac{1}{[E : \Q]} \sum_{\sigma\colon E \hookrightarrow \C} \log(\mathcal{L}(\sigma F)) 
& \leq \log(D+1) + h(F) \\
& \leq \frac{3}{2}\log(2\Card{S})+ \frac{6}{\h{S}}\left(\frac{3}{2}\log(2\Card{S})+2\Card{S}\m{S}\right) \\ 
& \leq \frac{12}{\h{S}} \left(\log(2\Card{S})+ \Card{S}\m{S}\right)
\end{align*}
because $D< L\leq 2\Card{S}-1$ and $\h{S}\leq 1$. By definition of $\h{S}$, we conclude 
\[\frac{1}{[E : \Q]} \sum_{\sigma\colon E \hookrightarrow \C} \log(\mathcal{L}(\sigma F)) 
\leq \frac{12\Card{S}}{\h{S}} \left(\frac{\h{S}}{24}\right)^3= \frac{\Card{S}\h{S}^2}{2^7\cdot 3^2}.\] 
Finally, Theorem \ref{thmET} applied to $Q=\sigma F$ gives \[\sum_{\sigma\colon E \hookrightarrow  \C} \left\vert Z_{\Delta,\sigma F}- \frac{\theta}{2\pi}D \right\vert^2 \leq 2^9\Card{S} \sum_{\sigma\colon E\hookrightarrow  \C}\log(\mathcal{L}(\sigma F))\]
since $\sum_\sigma \log\vert \sigma(f_Df_0)\vert\geq 0$ by the foregoing. 
The lemma follows.
\end{proof}

\begin{proof}[Proof of Theorem~\ref{thm-radial-angular}(2).]
Note that the left-hand side in Theorem~\ref{thm-radial-angular}(2) remains unchanged if we replace $\Q(S)$ with a finite extension. 
So, it is enough to prove Theorem~\ref{thm-radial-angular}(2) by replacing $\Q(S)$ with $E$. Let $\sigma\colon E \hookrightarrow \C$ be a $\Q$-embedding. 
The triangle inequality gives \[\left\vert Z_{\Delta,\sigma P_S} -\frac{\theta \Card{S}}{2\pi}\right\vert \leq \vert Z_{\Delta,\sigma P_S} - Z_{\Delta, \sigma F}\vert +  \left\vert Z_{\Delta,\sigma F}- \frac{\theta D}{2\pi}\right\vert + \left\vert \frac{\theta D}{2\pi} -\frac{\theta \Card{S}}{2\pi}\right\vert.\]
Recall that $P_S$ divides $F$. 
Thus, $Z_{\Delta,\sigma F} - Z_{\Delta, \sigma P_S}=Z_{\Delta, \sigma(F/P_S)}$. 
In particular, it is bounded from above by the degree of $F/P_S$, namely $D-\Card{S}$, and so by $L-1-\vert S\vert \leq \Card{S}\h{S}/6$. 
Thus, \[\frac{1}{[ E : \Q]}\sum_{\sigma\colon E \hookrightarrow \C} \left\vert Z_{\Delta,\sigma P_S}- \frac{\theta}{2\pi}\Card{S}\right\vert \leq \frac{\Card{S}\h{S}}{3} + \frac{1}{[ E : \Q]}\sum_{\sigma\colon E \hookrightarrow \C} \left\vert Z_{\Delta,\sigma F}- \frac{\theta}{2\pi}D \right\vert.\]
From Lemma \ref{lmm inter}, we have 
\begin{align*}
\frac{1}{[ E : \Q]}\sum_{\sigma\colon E \hookrightarrow \C} \left\vert Z_{\Delta,\sigma F}- \frac{\theta}{2\pi}D \right\vert 
&\leq \sqrt{\frac{1}{[ E : \Q]}\sum_{\sigma\colon E \hookrightarrow \C} \left\vert Z_{\Delta,\sigma F}- \frac{\theta}{2\pi}D \right\vert^2}\\[0.3cm] 
&\leq \frac{2\Card{S}\h{S}}{3}
\end{align*}
and the second assertion of Theorem~\ref{thm-radial-angular} follows. 
\end{proof}


\section{Convergence in distribution}
\label{main}
Let $S$ and $r$ be as in Theorem \ref{main thm}. By Theorem~\ref{thm-radial-angular}, we obtain that in average, the cardinality of $\sigma(S)\cap\Delta\cap \A_r$, with $\sigma\colon\Q(S) \hookrightarrow \C$ a $\Q$-embedding, is approximately $\theta\Card{S}/(2\pi)$ when $\m{S}$ and $\h{S}$ are small enough. 
This assertion is stronger when $\theta$ is small.
For this reason, it makes sense to cut $\A_r$ into a large number of small annulus sectors, apply Theorem~\ref{thm-radial-angular} to deduce that in average, there are around $\theta\Card{S}/(2\pi)$ elements of $\sigma(S)$ in each of these annulus sectors, and put it all the information together to conclude that in average, the set $\sigma(S)$ is equidistributed around the unit circle. 
Formalizing this process leads to the proof of Theorem \ref{main thm}(1). 
A slight modification of these arguments shows the second part. \\

We start with an easy lemma. Given $r>1$ and $t_0,t_1\in[0,1]$ with $t_0<t_1$, we consider the compact region
\[ V_{r,t_0,t_1}=\{\rho e^{2i\pi t}\;\vert\; \rho\in [1/r, r],\; t\in[t_0,t_1]\}. \]

\begin{lemma} \label{lmm conclusion 1}
Let $f:\C^* \to \R$ be a real function that is Lipschitz on $\A_r$ for some $r>1$, and let $T$ be a finite set of non-zero complex numbers. 
Choose $t_0,t_1\in[0,1]$ with $t_0<t_1$ and put $V=V_{r,t_0,t_1}$. 
If $\theta=t_1-t_0\in(0,1/2]$, then
\[ \left\vert\frac{1}{\Card{T}}\sum_{\beta\in T\cap V} f(\beta) - \int_{t_0}^{t_1} f(e^{2i\pi t})dt\right\vert
\leq 2r\pi\theta^2 \mathrm{Lip}_r(f) + \lVert f\rVert_{\infty,\A_r} \left\vert \frac{\Card{T\cap V}}{\Card{T}}-\theta\right\vert.  \]
\end{lemma}
\begin{proof}
Let $f^+$, resp. $f^-$, be the supremum, resp. infimum, of $f(z)$, where $z$ ranges over all elements of $V$. Then 
\[ \frac{f^-\Card{T\cap V}}{\Card{T}}-f^+\theta \leq \frac{1}{\Card{T}}\sum_{\beta\in T\cap V} f(\beta) - \int_{t_0}^{t_1} f(e^{2i\pi t})dt \leq  \frac{f^+\Card{T\cap V}}{\Card{T}}-f^-\theta.\]
We write
\[ \frac{f^\pm\Card{T\cap V}}{\Card{T}}-f^\mp\theta
=f^\pm\left(\frac{\Card{T\cap V}}{\Card{T}}-\theta\right)+(f^\pm-f^\mp)\theta.\]
Combining the last two displayed equations, then using the triangle inequality, we get 
\begin{multline}
\label{eq conclusion 1}
\left\vert\frac{1}{\Card{T}}\sum_{\beta\in T\cap V} f(\beta) - \int_{t_0}^{t_1} f(e^{2i\pi t})dt\right\vert \\ 
\leq (f^+-f^-)\theta+\lVert f\rVert_{\infty,\A_r} \left\vert \frac{\Card{T\cap V}}{\Card{T}}-\theta\right\vert.
\end{multline}
The fact that $V$ is compact implies that $f^+=f(x)$ and $f^-=f(y)$ for some $x,y\in V$. 
Each point $z\in V$ expresses as $r_ze^{i\alpha_z}$ with $r_z\in [1/r, r]$ and $\alpha_z\in [2\pi t_0,2\pi t_1]$. 
Hence, 
\begin{align*}
f^+-f^- & = \vert f(x)-f(y)\vert \leq \mathrm{Lip}_r(f) \vert r_xe^{i\alpha_x}-r_ye^{i\alpha_y}\vert \\ 
& = \mathrm{Lip}_r(f) \sqrt{r_x^2+r_y^2-2r_xr_y\cos(\alpha_y-\alpha_x)}.
\end{align*}
As $\theta\in(0,1/2]$, the cosine is decreasing on $[0, 2\pi \theta]$. Thus $\cos(\alpha_y-\alpha_x)\geq \cos(2\pi\theta)$, and so 
\begin{equation} \label{eq conclusion 2}
\vert f^+-f^-\vert \leq \mathrm{Lip}_r(f)\sqrt{2r^2(1-\cos(2\pi\theta))}
\end{equation}
since $r_x,r_y\leq r$. 
The lemma follows by combining \eqref{eq conclusion 1} and \eqref{eq conclusion 2}, then using the inequality $1-\cos(x)\leq x^2/2$, which is true for all $x\in\R$.  
\end{proof}

Fix from now a finite subset $S\subset \overline{\Q}^*$ as well as an integer $N\geq 2$.
For $x\in \R$ and $j\in\{0,\dots,N-1\}$, we denote by $\Delta_j(x)$ the sector based at the origin containing $0$ and all non-zero complex numbers whose argument belongs to $[x+2\pi j /N,x+ 2\pi (j+1)/N]$ up to $2k\pi$.
It is a sector of angle $2\pi/N$.
The Galois closure $\SGal$ of $S$ being finite, we can then find $x\in\R$ satisfying the following property: for all $\alpha\in \SGal$, there is (a unique) $j\in \{0,\dots,N-1\}$ such that $\alpha$ lies in the {\sl interior of} 
$\Delta_j(x)$.  We fix from now such an $x$ and in order to ease notation, we put $\Delta_j=\Delta_j(x)$.

Let $\sigma\colon\Q(S) \hookrightarrow \C$ be a $\Q$-embedding, and let $f: \C^*\to \C$ be a function that is Lipschitz on $\A_r$. We have 
\begin{multline}
\label{integral-decomposition}
\left\vert \int_{\C^*} fd \mu_{\sigma(S)} - \int_{\C^*} fd\lambda\right\vert
= \left\vert \frac{1}{\Card{S}}\sum_{\beta\in \sigma(S)} f(\beta) - \int_0^1 f(e^{2i\pi t})dt \right\vert \\ 
\leq\frac{\Card{\sigma(S)\backslash \A_r}}{\Card{S}}\Vert f\Vert_{\infty,\SGal\backslash \A_r}
+ \left\vert  \frac{1}{\Card{S}}\sum_{\beta\in \sigma(S)\cap\A_r} f(\beta) - \int_0^1 f(e^{2i\pi t})dt \right\vert.
\end{multline}
As each element of $\sigma(S)$ lies in $\Delta_j$ for a unique $j\in\{0,\ldots,N-1\}$, we get 
\begin{multline*}
\frac{1}{\Card{S}}\sum_{\beta\in \sigma(S)\cap\A_r} f(\beta) - \int_0^1 f(e^{2i\pi t})dt\\ 
=\sum_{j=0}^{N-1}\left(\frac{1}{\Card{S}}\sum_{\beta\in \sigma(S)\cap\Delta_j\cap\A_r} f(\beta) - \int_{j/N}^{(j+1)/N} f(e^{2i\pi t})dt\right).
\end{multline*}

The real and imaginary parts of $f$ are two real-valued functions which are Lipschitz on $\A_r$ since $f$ is. 
Moreover, their Lipschitz coefficients and suppremums on $\A_r$ are bounded from above by those of $f$.
Applying Lemma~\ref{lmm conclusion 1} twice to the real part (for the first time), then the imaginary part (for the second one) and with $t_0=j/N$, $t_1=(j+1)/N$ and $T=\sigma(S)$ in both cases, we conclude thanks to the triangle inequality that
\begin{multline*}
\left\vert\frac{1}{\Card{S}}\sum_{\beta\in \sigma(S)\cap\Delta_j\cap\A_r} f(\beta) - \int_{j/N}^{(j+1)/N} f(e^{2i\pi t})dt\right\vert \\ 
\leq \frac{4r\pi \mathrm{Lip}_r(f)}{N^2} + 2\lVert f\rVert_{\infty,\A_r} \left\vert \frac{\Card{\sigma(S)\cap\Delta_j\cap\A_r}}{\Card{S}}-\frac{1}{N}\right\vert.
\end{multline*}
We now infer that
\begin{multline} 
\label{second-term}
\left\vert \frac{1}{\Card{S}}\sum_{\beta\in \sigma(S)\cap\A_r} f(\beta) - \int_0^1 f(e^{2i\pi t})dt\right\vert\\ 
\qquad\leq\frac{4r\pi \mathrm{Lip}_r(f)}{N} 
+ 2\lVert f\rVert_{\infty,\A_r} \sum_{j=0}^{N-1}\left\vert \frac{\Card{\sigma(S)\cap\Delta_j\cap\A_r}}{\Card{S}}-\frac{1}{N}\right\vert.
\end{multline} 
Since $\Card{\sigma(S)\cap\Delta_j\cap\A_r}=\Card{\sigma(S)\cap\Delta_j}- \Card{\sigma(S)\cap\Delta_j\backslash A_r}$, 
the triangle inequality gives
\[\left\vert \frac{\Card{\sigma(S)\cap\Delta_j\cap\A_r}}{\Card{S}}-\frac{1}{N}\right\vert 
\leq\left\vert \frac{\Card{\sigma(S)\cap\Delta_j}}{\Card{S}}-\frac{1}{N}\right\vert
+\frac{\Card{\sigma(S)\cap\Delta_j\backslash A_r}}{\Card{S}}.\]
Since $\C$ is the union of all $\Delta_0, \dots, \Delta_{N-1}$ and each element of $\sigma(S)$ lies in the interior of $\Delta_j$ for a unique $j\in\{0,\dots,N-1\}$, we get
\[\sum_{j=0}^{N-1}\Card{\sigma(S)\cap\Delta_j\backslash\A_r}=\Card{\sigma(S)\backslash\A_r}.\]
Thus, 
\begin{equation}
\label{sum}
\sum_{j=0}^{N-1}\left\vert \frac{\Card{\sigma(S)\cap\Delta_j\cap\A_r}}{\Card{S}}-\frac{1}{N}\right\vert
\leq \frac{\Card{\sigma(S)\backslash\A_r}}{\Card{S}}+\sum_{j=0}^{N-1}\left\vert \frac{\Card{\sigma(S)\cap\Delta_j}}{\Card{S}}-\frac{1}{N}\right\vert.
\end{equation}
By~\eqref{integral-decomposition},~\eqref{second-term} and~\eqref{sum}, we get:
\begin{multline}
\label{bound-integral}
\left\vert \int_{\C^*} fd \mu_{\sigma(S)} - \int_{\C^*} fd\lambda\right\vert
\leq\frac{4r\pi \mathrm{Lip}_r(f)}{N}
+ (\Vert f\Vert_{\infty,\SGal\backslash \A_r}+ 2\lVert f\rVert_{\infty,\A_r})\frac{\Card{\sigma(S)\backslash \A_r}}{\Card{S}}\\
+ 2\lVert f\rVert_{\infty,\A_r}\sum_{j=0}^{N-1}\left\vert \frac{\Card{\sigma(S)\cap\Delta_j}}{\Card{S}}-\frac{1}{N}\right\vert.
\end{multline}

\begin{proof}[Proof of Theorem \ref{main thm}(1).] It now arises from Theorem~\ref{thm-radial-angular} (applied to the sector $\Delta=\Delta_j$ of angle $2\pi/N$) and from \eqref{bound-integral} that
\begin{multline*}
\frac{1}{[ \Q(S) : \Q]}\sum_{\sigma\colon\Q(S) \hookrightarrow \C}\left\vert \int_{\C^*} fd \mu_{\sigma(S)} - \int_{\C^*} fd\lambda\right\vert\\
\leq\frac{4r\pi \mathrm{Lip}_r(f)}{N}+(\Vert f\Vert_{\infty,\SGal\backslash \A_r}+ 2\lVert f\rVert_{\infty,\A_r})\frac{2\m{S}}{\log r}
+ 2N\lVert f\rVert_{\infty,\A_r}\h{S},
\end{multline*}
which is the first inequality of the theorem.
\end{proof}
\begin{proof}[Proof of Theorem \ref{main thm}(2).] We now prove the second assertion. Note that in the arithmetic means of Theorem~\ref{thm-radial-angular}, we may replace $\Q(S)$ with any of its finite field extensions. Thus, the first assertion of Theorem \ref{thm-radial-angular} easily implies that the set of $\Q$-embeddings $\sigma\colon L(S) \hookrightarrow \C$ satisfying
\begin{equation} \label{eq-radial2}
\frac{\Card{\sigma(S)\backslash \A_r}}{\Card{S} } \leq \frac{4[L:\Q]\m{S}}{\e\log r}
\end{equation} 
has cardinality at least $\big(1-\frac{\e}{2[L:\Q]}\big)[L(S) : \Q]$.  
Let $j\in\{0,\dots,N-1\}$. 
Similarly, the second assertion of Theorem~\ref{thm-radial-angular} (with $\Delta=\Delta_j$) provides at least $\big(1-\frac{\e}{2N[L:\Q]}\big)[L(S) : \Q]$ field embeddings $\sigma\colon L(S)\hookrightarrow\C$ for which
\begin{equation} \label{eq-angular2}
\left\vert \frac{\Card{\sigma(S)\cap\Delta_j}}{\Card{S}} - \frac{1}{N}\right\vert \leq \frac{2N[L:\Q]}{\e}\h{S}.
\end{equation}
Thus, there exists a set $Y=Y(S,r,N,\e,L)$ of $\Q$-embeddings $\sigma\colon L(S)\hookrightarrow\C$ with cardinality 
\[\Card{Y}\geq\left(1-\frac{\e}{[L:\Q]}\right)[L(S) : \Q]=[L(S) : \Q]- \e [L(S) : L]\] 
such that any $\sigma$ in this set satisfies \eqref{eq-radial2} and \eqref{eq-angular2} for all $j\in\{0,\dots,N-1\}$.  
We can moreover find a set $\Lambda\subseteq Y$ of $L$-embeddings with cardinality at least $(1-\e)[L(S) : L]$. 
Otherwise, the cardinality of $Y$ would be less than 
\[(1-\e)[L(S):L]+ ([L:\Q]-1)[L(S) : L]=[L(S) : \Q]- \e [L(S) : L],\] 
 a contradiction. We choose such a subset $\Lambda$ and we fix $\sigma\in \Lambda$. 
 
Theorem \ref{main thm}(2) follows since by~\eqref{bound-integral},~\eqref{eq-radial2} and~\eqref{eq-angular2}, we get
\begin{multline*}
\left\vert \int_{\C^*} fd \mu_{\sigma(S)} - \int_{\C^*} fd\lambda\right\vert
\leq\frac{4r\pi \mathrm{Lip}_r(f)}{N}\\
+(\Vert f\Vert_{\infty,\SGal\backslash \A_r}+ 2\lVert f\rVert_{\infty,\A_r})\frac{4[L:\Q]\m{S}}{\e\log r}
+ \frac{4[L:\Q]N^2}{\e}\lVert f\rVert_{\infty,\A_r}\h{S}
\end{multline*}
for all functions $f:\C^*\to \C$ that are Lipschitz on $\mathcal{A}_r$.
\end{proof}

\begin{proof}[Proof of Corollary~\ref{cor gen Bilu}.]
We only prove the second assertion, the proof of the first one following similar lines. Given an integer $n$, a $\Q$-embedding $\sigma\colon\Q(S_n)\hookrightarrow\C$ and a function $f\colon\C^*\to\C$, we let for short
\[u_{n,\sigma}(f)=\left\vert \int_{\C^*} f d \mu_{\sigma S_n} - \int_{\C^*} f d\lambda\right\vert.\]

Let $f: \C^* \to \C$ be a function satisfying \eqref{test}. Obviously, there exists $r>1$ such that $f$ is continuous on $\A_r$.
Since $f$ is not necessarily Lipschitz on $\A_r$, we use a standard density argument. Thanks to the Stone-Weierstrass theorem, we know that $f$ restricted to $\A_r$ is the uniform limit of a sequence of polynomial functions $(f_m)_m$. We extend $f_m$ to a function on $\C^*$ by setting $f_m=f$ on $\C^*\backslash\A_r$. Thus $f_m$ is Lipschitz on $\A_r$ for all $m$ and $(f_m)_m$ uniformly converges to $f$ on $\C^*$.

From our assumptions, we have $\h{S_n}\to 0$.
The second assertion of Theorem \ref{main thm} with $S=S_n$ and $N=[\h{S_n}^{\scriptscriptstyle-1/4}]$ proves that there is a set $\Lambda_n$ of $L$-embeddings $\sigma\colon L(S_n)\hookrightarrow\C$, depending only on $S_n, r, \varepsilon$ and $L$, but not on the test function $f$, such that $\Card{\Lambda_n}\geq (1-\e)[L(S_n):L]$ and
\begin{multline}\label{eq-U}
u_{n,\sigma}(f_m) \leq U_{n,m}=\frac{4r\pi \mathrm{Lip}_r(f_m)}{[\h{S_n}^{\scriptscriptstyle-1/4}]}
+(\Vert f_m\Vert_{\infty,S_n^{\Gal}\backslash \A_r}+ 2\lVert f_m\rVert_{\infty,\A_r})\frac{4[L:\Q]\m{S_n}}{\e\log r}\\
+ \frac{4[L:\Q][\h{S_n}^{\scriptscriptstyle-1/4}]^2}{\e}\lVert f_m\rVert_{\infty,\A_r}\h{S_n}
\end{multline} 
for all $m$. 
For each $n$, we choose one of these $L$-embeddings, say $\sigma_n$. We want to show that $u_{n,\sigma_n}(f)\to 0$, which would show the corollary.  

Let $m$ be an index. 
As $f=f_m$ outside $\mathcal{A}_r$, we deduce by assumption that $\m{S_n}\Vert f_m\Vert_{\infty,S_n^{\Gal}\backslash \A_r} \to 0$, and so \eqref{eq-U} leads to $U_{n,m}\to 0$ as $n\to +\infty$.  

Let $n$ be an index.
Using the reverse triangle inequality, then the triangle inequality, we get \[\left \vert u_{n,\sigma_n}(f_m)-  u_{n,\sigma_n}(f)\right\vert \leq \int_{\C^*} \vert f_m-f \vert d\mu_{\sigma_n S_n}+\int_{\C^*} \vert f_m-f \vert d\lambda\] for all $m$.
As $f_m$ uniformly converges to $f$ on $\C^*$, we deduce that $u_{n,\sigma_n}(f_m)$ uniformly converges to  $u_{n,\sigma_n}(f)$ as $m\to +\infty$.
The Moore-Osgood theorem for interchanging limits leads to $u_{n,\sigma_n}(f)\to 0$ as $n\to +\infty$.
\end{proof}

\bibliographystyle{plain}

\begin{thebibliography}{10} 

\bibitem{AmorosoDavidZannier}
F.~ Amoroso, S. David and U.~Zannier
\newblock {\em On fields with the property $(B)$},
\newblock {Proc. Amer. Math. Soc.}, $\mathbf{142}$ (2014), p. 893--1910.

\bibitem{AmorosoDvornicich}
F.~ Amoroso and R.~Dvornicich
\newblock {\em A Lower Bound for the Height in Abelian Extensions},
\newblock {J.~Number Theory}, $\mathbf{80}$ (2000), p. 260--272.

\bibitem{AmorosoMignotte}
F.~ Amoroso and M.~Mignotte
\newblock {\em On the distribution of the roots of polynomials},
\newblock {Ann. Inst. Fourier (Grenoble)}, $\mathbf{46}$ (1996), p. 1275--1291.

\bibitem{AmorosoZannier}
F.~ Amoroso and U.~Zannier
\newblock {\em A relative {D}obrowolski lower bound over abelian extensions},
\newblock {Ann. Scuola Norm. Sup. Pisa Cl. Sci. (4)}, $\mathbf{29}$ (2000), p. 711--727.

\bibitem{BakerHsia}
M.~ Baker and L.~Hsia
\newblock {\em Canonical heights, transfinite diameters, and polynomial dynamics},
\newblock {J. Reine Angew. Math.}, $\mathbf{585}$ (2005), p. 61--92.

\bibitem{BakerMasser}
M.~Baker and D.~Masser.
\newblock {\em Galois distribution on tori- A refinement, examples, and applications},
\newblock{Int. Math. Res. Not.}, $\mathbf{2022}$ (2022), https://doi.org/10.1093/imrn/rnac197

\bibitem{BakerPetsche}
M.~Baker and C.~Petsche.
\newblock {\em Global discrepancy and small points on elliptic curves},
\newblock{Int. Math. Res. Not.}, $\mathbf{2005}$ (2005), no.61, p. 3791--3834.

\bibitem{BakerRumely}
M.~ Baker and R.~ Rumely
\newblock {\em Equidistribution of small points, rational dynamics, and potential theory},
\newblock {Ann. Inst. Fourier}, $\mathbf{56}$ (2006), no.3, p. 625--688.

\bibitem{Bilu}
Y.~Bilu
\newblock  {\em Limit distribution of small points on algebraic tori},
\newblock {Duke Math. J.}, $\mathbf{89}$ (1997), no. 3, p. 465--476.

\bibitem{Bombieri} E.~Bombieri and U.~Zannier,  
\newblock  {\em A note on heights in certain infinite extensions of $\Bbb Q$.},
\newblock {Rend. Mat. Acc. Lincei (9)}, $\mathbf{12}$ (2001), p. 5--14.

\bibitem{Burgos}
G.~Burgos, J.~ Ignacio, P.~ Philippon, J.~ Rivera-Letellier and M.~ Sombra
\newblock  {\em The distribution of {G}alois orbits of points of small height in toric varieties},
\newblock {Amer. J. Math.}, $\mathbf{141}$ (2019), no. 2, p. 309--381.

\bibitem{ChambertLoir}
A.~ Chambert-Loir 
\newblock  {\em Mesures et \'equidistribution sur les espaces de Berkovich},
\newblock {J. Reine Angew. Math.}, $\mathbf{595}$ (2006), p. 215--235.

\bibitem{Dandrea}
C.~D'andrea, M.~ Narv\'{a}ez-Clauss and M.~ Sombra
\newblock  {\em Quantitative equidistribution of {G}alois orbits of small points in the {$N$}-dimensional torus},
\newblock {Algebra Number Theory}, $\mathbf{11}$ (2017), no. 7, p. 1627--1655.

\bibitem{Demarco}
L.~ DeMarco and N.M.~ Mavraki 
\newblock  {\em Variation of canonical height and equidistribution},
\newblock {Amer. J. Math.}, $\mathbf{142}$ (2020), no. 2, p. 443--473.

\bibitem{Fili}
J.~Doyle, P.~Fili and B.~Tobin
\newblock {\em Stochastic Equidistribution and Generalized Adelic Measures}, 
\newblock {https://arxiv.org/abs/2111.08905}

\bibitem{Erdos}
P.~ Erdős and P.~ Turán
\newblock  {\em On the Distribution of Roots of Polynomials},
\newblock {Ann. of Math. (2)}, $\mathbf{51}$ (1950), p. 105--119.

\bibitem{FavreLetelier}
C. Favre, J.Rivera-Letelier.
\newblock {\em \'{E}quidistribution quantitative des points de petite hauteur sur la droite projective},
\newblock{Math. Ann.}, $\mathbf{335}$ (2006), p. 311--361.

\bibitem{Fekete}
M. Fekete.
\newblock {\em \"{U}ber die {V}erteilung der {W}urzeln bei gewissen algebraischen {G}leichungen mit ganzzahligen {K}oeffizienten},
\newblock{Math. Z.}, $\mathbf{17}$ (1923), p. 228--249.

\bibitem{Gubler}
W.~ Gubler 
\newblock  {\em The Bogomolov conjecture for totally degenerate abelian varieties},
\newblock {Invent. Math.}, $\mathbf{169}$ (2007), no.2, p. 377-400.

\bibitem{Langevin}
M.~ Langevin
\newblock  {\em Minorations de la maison et de la mesure de {M}ahler de certains entiers alg\'{e}briques},
\newblock {C. R. Acad. Sci. Paris S\'{e}r. I Math.}, $\mathbf{303}$ (1986), no. 12, p. 523--526

\bibitem{Masser}
D.~Masser
\newblock {\em Auxiliary polynomials in number theory}, Cambridge University Press, Cambridge, 2016.

\bibitem{Mignotte}
M.~ Mignotte
\newblock  {\em Sur un th\'{e}or\`eme de {M}. {L}angevin},
\newblock {Acta Arith.}, $\mathbf{54}$ (1989), p. 81--86.

\bibitem{Petsche}
C.~ Petsche
\newblock  {\em A quantitative version of {B}ilu's equidistribution theorem},
\newblock {Int. J. Number Theory}, $\mathbf{1}$ (2005), no.2, p. 281--291.

\bibitem{RoyThunder}
D.~ Roy and J.~ Thunder
\newblock  {\em An absolute {S}iegel's lemma},
\newblock {J. Reine Angew. Math.}, $\mathbf{476}$ (1996), p. 1--26.

\bibitem{Rumely}
R.~ Rumely
\newblock  {\em On {B}ilu's equidistribution theorem},
\newblock {Contemp. Math.}, $\mathbf{237}$ (1999), p. 159--166.

\bibitem{Schinzel} A.~Schinzel,  
\newblock  {\em On the product of the conjugates outside the unit circle of an algebraic number.} 
\newblock{ Acta Arith.}, $\mathbf{24}$ (1973), p. 385--399. 

\bibitem{SUZ}
L.~Szpiro, E.~Ullmo and S.~ Zhang
\newblock  {\em \'{E}quir\'{e}partition des petits points},
\newblock {Invent. Math.}, $\mathbf{127}$ (1997), no. 2, p. 337--347.

\bibitem{Yuan}
X.~Yuan.
\newblock {\em Big line bundles over arithmetic varieties},
\newblock {Invent. Math.}, $\mathbf{173}$ (2008), p. 603--649.

\bibitem{Zhang2}
S.~ Zhang
\newblock  {\em Small points and adelic metrics},
\newblock {J. Algebraic Geom.}, $\mathbf{4}$ (1995), p. 281--300.

\bibitem{Zhang}
S.~ Zhang
\newblock  {\em Equidistribution of small points on abelian varieties},
\newblock {Ann. of Math. (2)}, $\mathbf{147}$ (1998), no. 1, p. 159--165.

\end{thebibliography}

\end{document}